\documentclass{amsart}
              [1997/02/02 v1.2e PROC Author Class]

\copyrightinfo{2006}%            % copyright year
  {American Mathematical Society}% copyright holder

\newtheorem{theorem}{Theorem}[section]
\newtheorem{corollary}[theorem]{Corollary}
\newtheorem{lemma}[theorem]{Lemma}
\newtheorem{proposition}[theorem]{Proposition}

\newtheorem{question}[theorem]{Question}

\begin{document}

\title{Namioka spaces and strongly Baire spaces}

\author{V.V.Mykhaylyuk}
\address{Department of Mathematics\\
Chernivtsi National University\\ str. Kotsjubyn'skogo 2,
Chernivtsi, 58012 Ukraine}
\email{vmykhaylyuk@ukr.net}

%    General info

%\date{January 1, 1994 and, in revised form, June 22, 1994.}

\commby{Ronald A. Fintushel}

%\dedicatory{This paper is dedicated to our authors.}


\begin{abstract}
A notion of strongly Baire space is introduced. Its definition is a transfinite development of some equivalent reformulation of the Baire space definition. It is shown that every strongly Baire space is a Namioka space and every $\beta-\sigma$-unfavorable space is a strongly Baire space.
\end{abstract}

\maketitle
\section{Introduction}

Investigation of the discontinuity points set of separately continuous functions, that is functions which are continuous with respect every variable, was beginning by R.Baire in [1] and was continued by many mathematicians. A Namioka's result [2] on the continuity points set of separately continuous functions, defined on the product of two spaces one of which is compact, has become a new impulse to development of this investigation.

A topological space $X$ is called {\it strongly countably complete}, if there exists a sequence $({\mathcal U}_n)_{n=1}^{\infty}$ of open covers of $X$ such that for every centered sequence $(F_n)^{\infty}_{n=1}$ of closed in $X$ sets $F_n$ such that for every $n\in \mathbb N$ there exists $U\in {\mathcal U}_n$
such that $F_n\subseteq U$, the intersection $\bigcap\limits_{n=1}^{\infty}F_n$ is nonempty.

\begin{theorem}[Namioka] Let $X$ be a strongly countably complete space, $Y$ be a compact space and $f:X\times Y \to \mathbb R$ be a separately continuous function. Then there exists an everywhere dense in $X$ $G_{\delta}$-set $A\subseteq X$ such that the function $f$ is continuous at every point of set $A\times Y$.
\end{theorem}

The following notions were introduced in [3].

A mapping $f:X\times Y \to \mathbb R$ {\it has the Namioka property} if there exists a dense in $X$ $G_{\delta}$-set $A\subseteq X$ such that $A\times Y\subseteq C(f)$, where by $C(f)$ we denote the set of the joint continuity points set of mapping $f$.

A Baire space $X$ is called {\it Namioka space}, if for every compact space $Y$ every separately continuous function $f:X\times Y\to \mathbb R$ has the Namioka property.

It was obtained in [4] that Namioka spaces are closely connected with topological games.

Let ${\mathcal P}$ be a system of subsets of topological space $X$. We define a $G_{\mathcal P}$-game on $X$, in which the players $\alpha$ and $\beta$ participate.
A nonempty open in $X$ set $U_0$ is the first move of $\beta$ and a nonempty open in $X$ set $V_1\subseteq U_0$ and set $P_1\in {\mathcal P}$ are the first move of $\alpha$. Further $\beta$ chooses a nonempty open in $X$ set $U_1\subseteq V_1$ and $\alpha$ chooses a nonempty open in $X$ set $V_2\subseteq U_1$ and a set $P_2\in {\mathcal P}$ and so on. The player $\alpha$ wins if $(\bigcap\limits_{n=1}^{\infty}V_n)\bigcap (\overline{\bigcup\limits_{n=1}^{\infty}P_n})\ne\O$. Otherwise
$\beta$ wins.

A topological space $X$ is called {\it $\alpha$-favorable in the $G_{{\mathcal P}}$-game} if $\alpha$ has a winning strategy in this game. A topological space $X$ is called {\it $\beta$-unfavorable in the $G_{{\mathcal P}}$-game} if $\beta$ has no winning strategy in this game. Clearly, any $\alpha$-favorable topological space $X$ is a $\beta$-unfavorable space.

In the case of ${\mathcal P}=\{X\}$ the game $G_{{\mathcal P}}$ is the classical Choquet game and $X$ is $\beta$-unfavorable in this game if and only if $X$ is a Baire space (see [3]). If ${\mathcal P}$ is the system of all finite (or one-point) subsets of $X$ then $G_{{\mathcal P}}$-game is called a {\it $\sigma$-game}.

J.~Saint-Raymond shows in [3] that for usage the topological games method in these investigations it is enough to require a weaker condition of $\beta$-unfavorability instead of the $\alpha$-favorability. He proved that every $\beta-\sigma$-unfavorable space is a Namioka space and generalized the Christensen result.

A further development of this technique leads to a consideration of another topological games which based on wider systems ${\mathcal P}$ of subsets of a topological space $X$.

Let $T$ be a topological space and ${\mathcal K}(T)$ be a collection of all compact subsets of $T$. Then $T$ is said to be {\it ${\mathcal K}$-countably-determined} if there exist a subset $S$ of the topological space ${\mathbb N}^{\mathbb N}$ and a mapping $\varphi:S\to {\mathcal K}(T)$ such that for every open in $T$ set $U\subseteq T$ the set $\{s\in S: \varphi(s)\subseteq U\}$ is open in $S$ and $T=\bigcup\limits_{s\in S} \varphi (s)$; and it is called {\it ${\mathcal K}$-analytical} if there exists such a mapping $\varphi$ for the set $S={\mathbb N}^{\mathbb N}$.

A set $A$ in a topological space $X$ is called {\it bounded} if for any continuous function $f:X\to \mathbb R$ the set $f(A)=\{f(a): a\in A\}$ is bounded.

The following theorem gives further generalizations of Saint-Raymond result.

\begin{theorem}\label{th:1.2}  Any $\beta$-unfavorable in $G_{\mathcal P}$-game topological space $X$ is a Namioka space if :

$(i)$  ${\mathcal P}$ is the system of all compact subsets of $X$ ({\bf M.~Talagrand} \cite{T});

$(ii)$ ${\mathcal P}$ is the system of all ${\mathcal K}$-analytical subsets of $X$ ({\bf G.~Debs} \cite{D});

$(iii)$ ${\mathcal P}$ is the system of all bounded subsets of $X$ ({\bf O.~Maslyuchenko} \cite{Ma});

$(iv)$  ${\mathcal P}$ is a system of all ${\mathcal K}$-countable-determined subsets of $X$ ({\bf V.~Rybakov} \cite{R}).
\end{theorem}

It is easy to see that $(iv)\Rightarrow (ii) \Rightarrow (i)$ and $(iii) \Rightarrow (i)$.

It was interesting in this connection to obtain the characterization of Namioka spaces using their internal structure without topoplogical games (analogously as in the classical definition od Baire space). In this paper we study the notion of strongly Baire space. It introduced in the form which is a transfinite generalization of one reformulation of the definition of Baire space. We show that every $\beta-\sigma$-unfavorable space is strongly Baire space and every strongly Baire space is a Namioka space. Moreover, we study sufficient conditions on a topological space $X$ which provide the metrizability of every compact $Y\subseteq C_p(X)$.

\section{Strongly Baire space and Namioka property}

It this section we introduce the notion of strongly Baire space and prove some results on relations between strongly Baire spaces and Namioka spaces and
$\beta-\sigma$-unfavorable spaces.

Let $\alpha$ be an ordinal. We say that {\it a net $(B_{\xi}:\xi<\alpha)$ of sets $B_{\xi}$ occupies a net $(A_{\xi}:\xi<\alpha)$ of sets $A_\xi$} if

$(i)$  $A_{\xi}\subseteq B_{\xi}$ for every $\xi<\alpha$;

$(ii)$ $A_{\beta}\subseteq \bigcup\limits_{\xi<\beta} B_{\xi}$ for every limited ordinal $\beta<\alpha$.

A topological space $X$ is called {\it strongly Baire}, if for every ordinal $\alpha$, an increasing  net $(A_{\xi}:\xi<\alpha)$ of closed in $X$ sets $A_{\xi}$ and an increasing net $(B_{\xi}:\xi<\alpha)$ of closed in $X$ sets $B_{\xi}$ which occupies the net $(A_{\xi}:\xi<\alpha)$ the following condition holds
$$
{\rm int}(\bigcup\limits_{\xi<\alpha} A_{\xi}) \subseteq \overline{\bigcup\limits_{\xi<\alpha} {\rm int}(B_{\xi})},
$$
where by ${\rm int}(C)$ we denote the interior of $C$. For $\alpha=\omega_0$ we obtain that
$${\rm int}(\bigcup\limits_{n\in \mathbb N} A_{n}) \subseteq \overline{\bigcup\limits_{n\in\mathbb N} {\rm int}(A_{n})}$$
for every sequence $(A_n)^{\infty}_{n=1}$ of closed in $X$ sets $A_n$. This is equivalent to the fact that $X$ is Baire. Hence the definition of strongly Baire space is a transfinite strengthening of the Baire condition.

We will use the following two proposition from [9]. The first of they illustrates a property of the dependence on some coordinates of continuous functions defined on compacts. The second proposition illustrates a relation between the dependence and the Namioka property of separately continuous functions. Together these propositions lead to consider the nets from the definition of strongly Baire space.

\begin{proposition}\label{p:2.1} Let $Y\subseteq {\mathbb R}^T$ be a compact space, $Z$ be a metric space, $f:Y\to Z$ be a separately continuous function, $\varepsilon \geq 0$ and set $S\subseteq T$ such that $|f(y')-f(y'')|_Z\leq\varepsilon$ for every $y',y''\in Y$ with $y'|_S=y''|_S$. Then for every $\varepsilon'>\varepsilon$ there exist a finite set $S_0\subseteq S$ and $\delta>0$ such that $|f(y')-f(y'')|_Z\leq \varepsilon'$ for every $y',y''\in Y$ with $|y'(s)-y''(s)|<\delta$ for every $s\in S_0$.
\end{proposition}

\begin{proposition}\label{p:2.2} Let $X$ be a Baire space, $Y\subseteq {\mathbb R}^T$ be a compact space, $f:X\times Y\to \mathbb R$ be a separately continuous function. Then the following conditions are equivalent:

$(i)$ $f$ has the Namioka property;

$(ii)$ for every open in $X$ nonempty set $U$ and real $\varepsilon >0$ there exist an open in $X$ nonempty set $U_0\subseteq U$ and at most countable set $S_0\subseteq T$ such that $|f(x,y')-f(x,y'')|\leq \varepsilon$ for every $x\in U_0$ and $y',y''\in Y$ with $y'|_{S_0}=y''|_{S_0}$.
\end{proposition}

\begin{theorem}\label{th:2.3} Every strongly Baire space is a Namioka space.
\end{theorem}

\begin{proof} Let $X$ be a strongly Baire space, $Y\subseteq {\mathbb R}^T$ be a compact space, $f:X\times Y\to \mathbb R$ be a separately continuous function, $U$ be an open in $X$ nonempty set and $\varepsilon>0$. According to Theorem 2.2, it is enough to prove that there exist an at most countable set $S_0\subseteq T$ and an open in $X$ nonempty set $U_0\subseteq U$ such that $|f(x,y')-f(x,y'')|\leq \varepsilon$ for every $x\in U_0$ and $y',y''\in Y$ with $y'|_{S_0}=y''|_{S_0}$.

Suppose the contrary, that is for every at most countable set $S\subseteq T$ and open in $X$ nonempty set $U'\subseteq U$ there exist $x\in U'$ and $y',y''\in Y$ with $y'|_{S}=y''|_{S}$ such that $|f(x,y')-f(x,y'')|> \varepsilon$. Let $\omega$ be the first ordinal with the cardinality $|T|$, $T=\{t_{\alpha}:\alpha<\omega\}$ and $(\varepsilon_n)^{\infty}_{n=0}$ be a strictly increasing sequence of reals $\varepsilon_n>0$ which tends to $\varepsilon$.

For every $\alpha<\omega$ we denote by $A^{(1)}_{\alpha}$ the set of all $x\in \overline{U}$ such that $|f(x,y')-f(x,y'')|\leq \varepsilon_0$ for every $y',y''\in Y$ with $y'(t_{\xi})=y''(t_{\xi})$ for $\xi<\alpha$ and by $B^{(1)}_{\alpha}$ we denote the set of all $x\in\overline{U}$ such that $|f(x,y')-f(x,y'')|\leq \varepsilon_1$ for every $y',y''\in Y$ with $y'(t_{\xi}=y''(t_{\xi})$ for $\xi\leq\alpha$. It follows from the continuity of $f$ with respect to $x$ that all sets $A_{\alpha}^{(1)}$ and $B^{(1)}_{\alpha}$ are closed, besides $A^{(1)}_{\alpha}\subseteq B^{(1)}_{\alpha}$ for every $\alpha<\omega$ and the sequences $(A^{(1)}_{\alpha}:\alpha<\omega)$ and $(B^{(1)}_{\alpha}:\alpha<\omega)$ are increasing. Moreover, it follows from Proposition 2.1 that for every limited ordinal $\alpha<\omega$ and every $x\in A^{(1)}_{\alpha}$ there exists an ordinal $\xi<\alpha$ such that $x\in B^{(1)}_{\xi}$, that is $A^{(1)}_{\alpha}\subseteq \bigcup\limits_{\xi<\alpha}B^{(1)}_{\xi}$. Thus, the sequence $(B^{(1)}_{\alpha}:\alpha<\omega)$ occupies the sequence $(A^{(1)}_{\alpha}:\alpha<\omega)$. It follows from Proposition 2.1 that $U\subseteq\bigcup\limits_{\alpha<\omega}A^{(1)}_{\alpha}$.
Since $X$ is strongly Baire, $\bigcup\limits_{\alpha<\omega}{\rm int}(B^{(1)}_{\alpha})\ne\O$. Therefore there exist an ordinal $\beta_1<\omega$ and an open in  $X$ nonempty set $U_1\subseteq U$ such that $\overline{U_1}\subseteq B^{(1)}_{\beta}$. Note that $\beta_1$ is a uncountable ordinal. Really, otherwise the at most countable set $S=\{t_{\alpha}:\alpha\leq\beta_1\}$ and the open set $U_1$ such that $|f(x,y')-f(x,y'')|\leq \varepsilon_1\leq\varepsilon$ for every $x\in U_1$ and $y',y''\in Y$ with $y'|_S=y''|_S$, a contradiction.

Put $S_1=\{t_{\alpha}:\alpha\leq\beta_1\}$. Let $\gamma_1$ be the first ordinal with the cardinality $|S_1|$ and $S_1=\{s^{(1)}_{\alpha}:\alpha<\gamma_1\}$. Clearly that $\gamma_1\leq\beta_1$. Note that $|f(x,y')-f(x,y'')|\leq \varepsilon_1$ for every $x\in\overline{U}_1$ and $y',y''\in Y$ with $y'|_{S_1}=y''|_{S_1}$. For every $\alpha<\gamma_1$ we denote by $A^{(2)}_{\alpha}$ the set of all $x\in \overline{U}_1$ such that $|f(x,y')-f(x,y'')|\leq \varepsilon_2$ for every $y',y''\in Y$ with $y'(s^{(1)}_{\xi})=y''(s^{(1)}_{\xi})$ for $\xi<\alpha$ and by $B^{(2)}_{\alpha}$ we denote the set of all $x\in\overline{U}_1$ such that $|f(x,y')-f(x,y'')|\leq \varepsilon_3$ for every $y',y''\in Y$ with $y'(s^{(1)}_{\xi})=y''(s^{(1)}_{\xi})$ for $\xi\leq\alpha$.
All sets $A^{(2)}_{\alpha}$ and $B^{(2)}_{\alpha}$ for $\alpha<\gamma_1$ are closed. It follows from 2.1 that the sequence $(B^{(2)}_{\alpha}:\alpha<\gamma_1)$ occupies the sequence $(A^{(2)}_{\alpha}:\alpha<\gamma_1)$ and $U_1\subseteq\bigcup\limits_{\alpha<\gamma_1}A^{(2)}_{\alpha}$. Since $X$ is strongly Baire, there exist an ordinal $\beta_2<\gamma_1$ and an open in $X$ nonempty set $U_2\subseteq U_1$ such that $\overline{U}_2\subseteq B^{(2)}_{\beta_2}$. We can to show analogously as for the ordinal $\beta_1$ that $\beta_2$ is a uncountable ordinal.

Put $S_2=\{s^{(1)}_{\alpha}:\alpha<\beta_2\}$. Let $\gamma_2$ be the first ordinal of the cardinality $|S_2|$ and $S_2=\{s^{(2)}_{\alpha}:\alpha<\gamma_2\}$. Clearly that $\gamma_2\leq\beta_2$. For every $\alpha<\gamma_2$ we denote by $A^{(3)}_{\alpha}$ the set of all $x\in\overline{U}_2$ such that $|f(x,y')-f(x,y'')|\leq \varepsilon_4$ for every $y',y''\in Y$ with $y'(s^{(2)}_{\xi})=y''(s^{(2)}_{\xi})$ for $\xi<\alpha$ and by $B^{(3)}_{\alpha}$ we denote the set of all $x\in\overline{U}_2$ such that $|f(x,y')-f(x,y'')|\leq \varepsilon_5$ for every $y',y''\in Y$ with $y'(s^{(2)}_{\xi})=y''(s^{(2)}_{\xi})$ for $\xi\leq\alpha$. It follows from Proposition 2.1 that the sequence $(B^{(3)}_{\alpha}:\alpha<\gamma_2)$ occupies the sequence $(A^{(3)}_{\alpha}:\alpha<\gamma_2)$ and $U_2\subseteq\bigcup\limits_{\alpha<\gamma_2}A^{(3)}_{\alpha}$. Therefore there exist an at most countable ordinal $\beta_3<\gamma_2$ and an open in $X$ nonempty set $U_3\subseteq U_2$ such that $\overline{U}_3\subseteq B^{(3)}_{\beta_3}$.

Continuing this process to infinity we obtain a sequence of ordinals
$$\beta_1\geq\gamma_1>\beta_2
\geq\gamma_2>\beta_3\geq\gamma_3>\dots,
$$
a contradiction.
\end{proof}

\begin{theorem}\label{th:2.4} Every $\beta-\sigma$-unfavorable space is a strongly Baire space.\end{theorem}

\begin{proof} Let a topological space $X$ is not strongly Baire. That is there exist an ordinal $\omega$ and increasing nets $(A_{\alpha}:\alpha<\omega)$ and
$(B_{\alpha}:\alpha<\omega)$ of closed in $X$ sets $A_{\alpha}$ and $B_{\alpha}$ such that the net $(B_{\alpha}:\alpha<\omega)$ occupies the net $(A_{\alpha}:\alpha<\omega)$ and $$U_0={\rm int}(\bigcup\limits_{\alpha<\omega}A_{\alpha})\setminus \overline{\bigcup\limits_{\alpha<\omega}{\rm int}(B_{\alpha})}\ne \O.$$

Clearly that $\omega$ is a limited ordinal. According to [10, Theorem 10, p.282] the ordinal $\omega$ is confinal to the first ordinal $\omega'$ with the same confinality. If $\omega'=\omega_0$, then $X$ is not Baire. That is the space $X$ is a $\beta$-unfavorable space, in particular, $X$ is $\beta-\sigma$-unfavorable.

We consider the case of $\omega'>\omega_0$. We describe an winner strategy $\tau$ for the player $\beta$ in $\sigma$-game. Let $\alpha_0<\omega$, $U_0$ is the first move of $\beta$, $V_1\subseteq U_0$ be an open in $X$ nonempty set and $x_1\in X$. If $x_1\in U_0$, then $x_1\in \bigcup\limits_{\alpha<\omega}A_{\alpha}$. Therefore there exists an ordinal $\alpha_1<\omega$ such that $x_1\in A_{\alpha_1}$. If $x_1\not \in U_0$, then we put $\alpha_1=1$. The closed set $B_{\alpha_1}$ is nowhere dense in $U_0$. Therefore the open set $U_1=\tau(U_0,V_1,x_1)=V_1\setminus B_{\alpha_1}$ is nonempty.

Let $V_2\subseteq U_1$ be an open in $X$ nonempty set and $x_2\in X$. If $x_2\in U_0$, then we choose an ordinal $\alpha_2$ such that $\alpha_1<\alpha_2<\omega$ and  $x_2\in A_{\alpha_2}$. If $x_2\not \in U_0$, then we put $\alpha_2=\alpha_1+1$. Now we put $U_2=\tau(U_0,V_1,x_1,U_1,V_2,x_2)= V_2\setminus B_{\alpha_2}$.

Continuing this process to infinity we obtain an strictly increasing sequence $(\alpha_n)^{\infty}_{n=1}$ of ordinals $\alpha_n$, an sequence $(x_n)^{\infty}_{n=1}$ of $x_n\in X$ and decreasing sequences $(U_n)^{\infty}_{n=0}$ and $(V_n)^{\infty}_{n=1}$ of open in $X$ nonempty sets $U_n$ and $V_n$ respectively such that $V_n\subseteq U_{n-1}$, $U_n=V_{n-1}\setminus B_{\alpha_n}$, where $\alpha_n=\alpha_{n-1}+1$, if $x_n\not \in U_0$, and $\alpha_n>\alpha_{n-1}$ such that $x_n\in A_{\alpha_n}$, if $x_n\in U_0$.

We show that $(\bigcap\limits^{\infty}_{n=0}U_n)\bigcap\overline{\{x_n:n\in\mathbb N\}}=\O$. Put $\gamma={\rm sup}\,\alpha_n$. Since $\omega'>\omega_0$, $\gamma<\omega$. We consider the sets $A=\{x_n:n\in\mathbb N, x_n\not\in U_0\}$ and $B=\{x_n:n\in\mathbb N, x_n\in U_0\}$. Note that $A\cap U_0=\O$. Therefore
$\overline{A}\cap(\bigcap\limits^{\infty}_{n=0}U_n)=\O$. Since the sequence $(A_{\alpha}:\alpha<\omega)$ is increasing, according to the choice of ordinals $\alpha_n$ we have $B\subseteq A_{\gamma}$. Recall that the sequence $(B_{\alpha}:\alpha<\omega)$ occupies the sequence $(A_{\alpha}:\alpha<\omega)$. Therefore for limited ordinal $\gamma$ we have
$$
\overline{B}\subseteq A_{\gamma}\subseteq\bigcup\limits_{\alpha<\gamma}B_{\alpha}=
\bigcup\limits_{n=1}^{\infty}B_{\alpha_n}.
$$
On other hand, according to the construction, we have $U_n\cap B_{\alpha_n}=\O$. Therefore $(\bigcap\limits_{n=0}^{\infty}U_n)\cap(\bigcup\limits^{\infty}_{n=1}B_{\alpha_n})=\O$.
Hence, $\overline{B}\cap(\bigcap\limits_{n=0}^{\infty}U_n)=\O$.

Thus, $(\bigcap\limits^{\infty}_{n=0}U_n)\cap\overline{\{x_n:n\in\mathbb
N\}}=\O$. Hence, the strategy $\tau$ is a winner strategy for $\beta$ in $\sigma$-game and $X$ is a $\beta-\sigma$-favorable space.
\end{proof}

\section{Properties of strongly Baire spaces}

In this section we investigate properties of strongly Baire spaces which related with the calibers of these spaces. The proof of the following statement is obvious.

\begin{proposition}\label{p:3.1} Нехай ${\rm int}(\bigcup\limits_{\xi<\alpha} A_{\xi}) \subseteq \overline{\bigcup\limits_{\xi<\alpha} {\rm int}(A_{\xi})}$ for every increasing sequence $(A_{\xi}:\xi<\alpha)$ of closed in topological space $X$ sets $A_{\xi}$. Then $X$ is a strongly Baire space.
\end{proposition}

Recall (see [11 p.16]) that a cardinal $\aleph$ is called {\it a caliber of the topological space $X$}, if for every family $(U_{\alpha}:\alpha\in A)$ of nonempty open in $X$ sets $U_{\alpha}$ з $|A|=\aleph$ there exists a set $B\subseteq A$ such that $|B|=\aleph$ and $\bigcap\limits_{\alpha\in B}U_{\alpha}\ne\O$.

\begin{proposition}\label{p:3.2} Let a Baire space $X$ such that every regular uncountable cardinal is a caliber of $X$. Then ${\rm int}(\bigcup\limits_{\xi<\alpha} A_{\xi}) \subseteq \overline{\bigcup\limits_{\xi<\alpha} {\rm int}(A_{\xi})}$ for every increasing sequence $(A_{\xi}:\xi<\alpha)$ of closed in $X$ sets $A_{\xi}$, in particular, $X$ is strongly Baire.
\end{proposition}

\begin{proof} Let $(A_{\xi}:\xi<\alpha)$ be a increasing sequence of closed in $X$ sets $A_{\xi}$. If $\alpha$ is not limited ordinal, that is $\alpha=\beta+1$, then
$$
{\rm int}(\bigcup\limits_{\xi<\alpha} A_{\xi}) = {\rm int}(A_{\beta}) \subseteq \overline{\bigcup\limits_{\xi<\alpha} {\rm int}(A_{\xi})}.
$$

Let $\alpha$ is a limited ordinal. According to [10, Theorem 10, p.282], the cardinal $\alpha$ is confinal to the first ordinal $\omega$ with the same confinality. This implies that the cardinal $\aleph=|\omega|$ is regular.

We choose an strictly increasing sequence $(\xi_{\gamma}:\gamma<\omega$ of ordinals $\xi_{\gamma}$ such that $\sup\limits_{\gamma<\omega}\,\xi_{\gamma}=\alpha$. For every $\gamma<\omega$ we put $F_{\gamma}=A_{\xi_{\gamma}}$. Since the sequence $(A_{\xi}:\xi<\alpha)$ is increasing, the sequence $(F_{\gamma}:\gamma<\omega)$ is increasing too, $\bigcup\limits_{\xi<\alpha}A_{\xi}=\bigcup\limits_{\gamma<\omega}F_{\gamma}$ and $\bigcup\limits_{\xi<\alpha}{\rm int}(A_{\xi}) = \bigcup\limits_{\gamma<\omega}{\rm int}(F_{\gamma})$.

If $\omega=\omega_0$, then it follows from the fact that $X$ is Baire that
$${\rm int}(\bigcup\limits_{\gamma<\omega_0}F_{\gamma})\subseteq
\overline{\bigcup\limits_{\gamma<\omega_0}{\rm int}(F_{\gamma})},$$
that is ${\rm int}(\bigcup\limits_{\xi<\alpha}A_{\xi})\subseteq \overline{\bigcup\limits_{\xi<\alpha}{\rm int}(A_{\xi})}$.

Let $\omega$ is an uncountable ordinal. Suppose that $U=\bigcup\limits_{\gamma<\omega}{\rm int}(F_{\gamma})\not\subseteq \overline{\bigcup\limits_{\gamma<\omega}{\rm int}(F_{\gamma})}$. Then $U\not\subseteq F_{\gamma}$, that is $U_{\gamma}=U\setminus F_{\gamma}\ne \O$ for every $\gamma<\omega$. Since $\aleph$ is a caliber of $X$, for the family $(U_{\gamma}:\gamma<\omega)$ there exists $x_0\in X$ such that the set $\Gamma=\{\gamma<\omega: x_0\in U_{\gamma}\}$ has the cardinality $\aleph$. Since $\omega$ is the first ordinal of the cardinality $\aleph$, $\sup\Gamma=\omega$. Note that $(U_{\gamma}:\gamma<\omega)$ is decreasing sequence, therefore $\bigcap\limits_{\gamma<\omega}U_{\gamma}=\bigcap\limits_{\gamma\in\Gamma}U_{\gamma}\ni x_0$. Thus, $\bigcap\limits_{\gamma<\omega}U_{\gamma}\ne\O$, that is $U\not\subseteq\bigcup\limits_{\gamma<\omega}F_{\gamma}$, in particular, $x_0\in U$ and $x_0\not\in\bigcup\limits_{\gamma<\omega}F_{\gamma}$. But this contradicts to $U={\rm int}(\bigcup\limits_{\gamma<\omega}F_{\gamma})$.
\end{proof}

It easy to see that for every separable space $X$ every infinite regular cardinal is a caliber of $X$. Moreover, according to [11, Theorem 0.3.13, p. 16], the product of multipliers with caliber $\aleph$ has the caliber $\aleph$ too. Hence, the following result is true.

\begin{proposition} \label{p:3.3} Let a topological product $X=\prod\limits_{t\in T}X_t$ is a Baire space, besides all spaces $X_t$ are separable. Then $X$ is strongly Baire.
\end{proposition}

Further, we shall use the following auxiliary statement (see [12, p.185]), which often be used for the investigation of product properties.

\begin{lemma}[Shanin]\label{l:3.4} Let $\aleph$ be an uncountable regular cardinal, $(T_{\gamma}:\gamma\in \Gamma)$ be a family of finite sets $T_{\gamma}$, moreover $|\Gamma|=\aleph$. Then there exist a set $\Delta\subseteq\Gamma$ and a finite set $S\subseteq \bigcup\limits_{\gamma\in\Gamma}T_{\gamma}$ such that $|\Delta|=\aleph$ and $T_{\beta}\cap T_{\gamma}=S$ for every distinct $\beta, \gamma \in \Delta$.
\end{lemma}

\begin{proposition}\label{p:3.5} Let $(X_t:t\in T)$ be a family of separable metric spaces $(X_t,|\cdot-\cdot|_t)$, $X\subseteq \prod\limits_{t\in T}X_{t}$ be a Baire space, which is dense in the space $Y= \prod\limits_{t\in T}X_{t}$ with the topology of uniform convergence on $T$. Then $X$ is a strongly Baire space.
\end{proposition}

\begin{proof} Let $\aleph$ be an uncountable regular cardinal, $(V_{\gamma}:\gamma\in \Gamma)$ be a family of nonempty open in $X$ basic sets $V_{\gamma}$, besides $|\Gamma|=\aleph$ and $(U_{\gamma}:\gamma\in \Gamma)$ be a family of nonempty open in $\prod\limits_{t\in T}X_{t}$ basic sets $U_{\gamma}=\prod\limits_{t\in T}U_{\gamma}^{(t)}$ such that $V_{\gamma}=U_{\gamma}\bigcap\prod\limits_{t\in T}X_{t}$ for every $\gamma\in \Gamma$. Put $T_{\gamma}=\{t\in T:U_{\gamma}^{(t)}\ne X_t\}$ for every $\gamma\in \Gamma$. Taking into account the regularity of $\aleph$ and lemma 3.4 we can propose that all finite sets $T_{\gamma}$ have the same cardinality and $T_{\gamma}\cap T_{\beta}= S$ for every distinct $\gamma, \beta \in \Gamma$ and some finite set $S\subseteq T$. For every $\gamma\in \Gamma$ we choose points $x_{\gamma}^{(t)}\in X_t$ for $t\in T_{\gamma}$ and real $\delta_{\gamma}> 0$ such that $\{x_t\in X_t: |x_t-x_{\gamma}^{(t)}|_t<\delta_{\gamma}\}\subseteq U_{\gamma}^{(t)}$. Since $\aleph$ is an uncountable regular cardinal, there exist $n\in\mathbb N$ and $\Gamma'\subseteq \Gamma$ such that $|\Gamma'|=\aleph$ and $\delta_{\gamma}\geq \frac{2}{n}$ for every $\gamma\in\Gamma'$. In separable metric space $\tilde{X}= \prod\limits_{s\in S}X_{s}$ we find a point $\tilde x\in \tilde X$ such that $|\tilde x(s)-x_{\gamma}^{(s)}|_s<\frac{1}{n}$ for every $s\in S$ and $\gamma\in \Gamma''$, where $\Gamma'' \subseteq \Gamma'$ is a set with $|\Gamma''|=\aleph$.

We put $T'=\bigcup\limits_{\gamma \in \Gamma''}S_{\gamma}$, where $S_{\gamma}=T_{\gamma}\setminus S$. Note that $T'=\O$ or $|T'|=\aleph$. If $T'=\O$, that is $T_{\gamma}=S$ for every $\gamma\in \Gamma''$, then we choose a point $y_0\in Y$ such that $y_0(s)=\tilde{x}(s)$ for every $s\in S$. Since the set $X$ is dense in $Y$, there exists a point $x_0\in X$ such that $|x_0(t)-y_0(t)|_t<\frac{1}{n}$ for every $t\in T$. Then for every $\gamma\in \Gamma''$ and $t\in T_{\gamma}$ we have
$$
|x_0(t)-x_{\gamma}^{(t)}|_t\leq |x_0(t)-y_0(t)|_t +
|y_0(t)-x_{\gamma}^{(t)}|_t =
$$
$$
=|x_0(t)-y_0(t)|_t + |\tilde x(t)-x_{\gamma}^{(t)}|_t< \frac{1}{n}
+ \frac{1}{n} \leq \delta_{\gamma}.
$$
Thus, $x_0(t)\in U_{\gamma}^{(t)}$ for every $\gamma\in \Gamma''$ and $t\in T_{\gamma}$, that is $x_0\in V_{\gamma}$ for every $\gamma \in \Gamma''$.

Let $|T'|=\aleph$. We choose a point $y_0\in Y$ such that $y_0(s)=\tilde{x}(s)$ for every $s\in S$ and $y_0(s)=x_{\gamma}^{(s)}$ for every $\gamma \in \Gamma''$ and $s\in S_{\gamma}$. Then, analogously as in the previous case, for some point $x_0\in X$ such that $|x_0(t)-y_0(t)|_t<\frac{1}{n}$ for every $t\in T$, we have $x_0\in V_{\gamma}$ for every $\gamma \in \Gamma''$.

Thus, every uncountable cardinal $\aleph$ is the caliber of the Baire space $X$. According to Proposition 3.2,  $X$ is a strongly Baire space.
\end{proof}

\begin{corollary}\label{c:3.6}  Let $X\subseteq [0,1]^T$ be a Baire space which is dense in the space $Y=[0,1]^T$ with the topology of the uniform convergence on $T$. Then $X$ is a strongly Baire space.
\end{corollary}

\begin{proposition} \label{p:3.7} Let $X$ be a strongly Baire space, $Y$ be a topological space, $f:X\to Y$ be a continuous surjective mapping such that for every nowhere dense in $Y$ set $B$ the set $f^{-1}(B)$ is nowhere dense in $X$. Then $Y$ is a strongly Baire space.
\end{proposition}

\begin{proof} Let $\alpha$ be a limited ordinal, $(A_{\xi}:\xi<\alpha)$ and $(B_{\xi}:\xi<\alpha)$ be increasing sequences of closed in $Y$ sets $A_{\xi}$ and $B_{\xi}$ such that $(B_{\xi}:\xi<\alpha)$ occupiers $(A_{\xi}:\xi<\alpha)$.
Show that
$$
{\rm int}(\bigcup\limits_{\xi<\alpha} A_{\xi}) \subseteq \overline{\bigcup\limits_{\xi<\alpha} {\rm int}(B_{\xi})}.
$$

We take a point $y_0\in {\rm int}(\bigcup\limits_{\xi<\alpha} A_{\xi})$ and a closed in $Y$ neighborhood $V$ of $y_0$. For every $\xi<\alpha$ we put 
$\tilde{A}_{\xi}=f^{-1}(A_{\xi}\cap V)$, $\tilde{B}_{\xi}=f^{-1}(B_{\xi}\cap V)$. Clearly $\tilde{A}_{\xi}$ and $\tilde{B}_{\xi}$ are closed in $X$, moreover $\tilde{A}_{\xi}\subseteq \tilde{B}_{\xi}$ for every $\xi<\alpha$. Besides, for every limited ordinal $\beta<\alpha$ we have 
$$
\tilde{A}_{\beta}=f^{-1}(A_{\beta}\cap V)\subseteq f^{-1}(\bigcup\limits_{\xi<\beta}(B_{\xi}\cap V))=\bigcup\limits_{\xi<\beta}f^{-1}(B_{\xi}\cap
V)=\bigcup\limits_{\xi<\beta}\tilde{B}_{\xi}.
$$
Thus, the sequence $(\tilde{B}_{\xi}:\xi<\alpha)$ occupies the sequence $(\tilde{A}_{\xi}:\xi<\alpha)$. Since a mapping $f$ is continuous and the set 
$\bigcup\limits_{\xi<\alpha}(A_{\xi}\cap V)$ is a neighborhood of $y_0$, the set $f^{-1}(\bigcup\limits_{\xi<\alpha}(A_{\xi}\cap V))=\bigcup\limits_{\xi<\alpha}\tilde{A}_{\xi}$ is a neighborhood of every point $x\in f^{-1}(y_0)$, in particular, ${\rm int}(\bigcup\limits_{\xi<\alpha}\tilde{A}_{\xi})\ne\O$. Taking into account that $X$ is a strongly Baire space, we obtain that $\bigcup\limits_{\xi<\alpha}{\rm int}(\tilde{B}_{\xi})\ne\O$, that is there exists $\gamma<\alpha$ such that ${\rm int}(\tilde{B}_{\gamma})\ne\O$. The set $\tilde{B}_{\gamma}$ is not nowhere dense in $X$, therefore the Proposition conditions imply that the set $B_{\gamma}\cap V$ is nowhere dense in $Y$. Since $B_{\gamma}\cap V$ is closed, ${\rm int}(B_{\gamma}\cap V)\ne\O$. Thus, ${\rm int}(B_{\gamma})\cap V\ne \O$ and $y_0\in \overline{\bigcup\limits_{\xi<\alpha}{\rm int}(B_{\xi})}$.
\end{proof}

\section{Metrizable compacts in space of continuous functions}

In this section we investigate sufficient conditions on a topological space $X$ for the metrizability of every compacts $Y\subseteq C_p(X)$.

Recall that {\it a topological space $X$ has countable chain condition} if every disjoint system of open sets is at most countable.

\begin{theorem}\label{th:4.1} Let $X$ be a strongly Baire space with the countable chain condition, in particular, a Baire space for which every regular uncountable cardinal is the caliber. Then every compact $Y\subseteq C_p(X)$ is metrizable.\end{theorem}

\begin{proof} Firstly, we note that for a completely regular space $X$ this theorem follows from Theorem 2.3 and [13, Theorem 2 and Proposition 5].

In general case for a Namioka space $X$ with the countable chain condition and a compact $Y\subseteq C_p(X)$, using Theorem 2.2 it easy to construct an at most countable set $A\subseteq X$ such that $y'|_A\ne y''|_A$ for every distinct $y',y''\in Y$. This implies the metrizability of $Y$.
\end{proof}

On other hand, the following result is true. It is an analog of Theorem 3.1 from [14] where a similar relation between the weight of a compact space $X$ and the caliber of $C_p(X)$ is obtained.

\begin{theorem}\label{th:4.2} Let $X$ be a topological space, $Y\subseteq C_p(X)$ be a nonmetrizable compact and $\aleph={\rm cof}(w(Y))>\aleph_0$. Then $\aleph$ is not caliber of $X$.
\end{theorem}

\begin{proof} Let $T$ be a set with $|T|=w(Y)$, $Z\subseteq\mathbb R^T$ be a compact such that there exists a homeomorphism $\varphi: Z\to Y$.

We consider a separately continuous function $f:X\times Z\to\mathbb R$, $f(x,z)=\varphi(z)(x)$. Let $\omega$ is the first ordinal with the cardinality $\aleph$ and $(T_{\alpha}:\alpha<\omega)$ is an increasing sequence of sets $T_{\alpha}\subseteq T$ such that $\bigcup\limits_{\alpha<\omega}T_{\alpha}=T$ and $|T_{\alpha}|<|T|$ for every $\alpha<\omega$. For each $\alpha<\omega$ we put 
$$A_{\alpha}=\{x\in X: f(x,z')=f(x,z'')\mbox{\,\,for\,\,every\,\,}z', z''\in Z
\mbox{\,\,with\,\,}z'|_{T_{\alpha}}=z''|_{T_{\alpha}}\}.$$ 
It follows from the continuity of $f$ with respect to $x$ that all sets $A_{\alpha}$ are closed in $X$. Since $f$ is continuous with respect to the second variable, for every $x\in X$ there exists an at most countable set $S\subseteq T$ such that $f(x,z')=f(x,z'')$ for every $z',z''\in Z$ with $z'|_S=z''|_S$. Therefore $X=\bigcup\limits_{\alpha<\omega}A_{\alpha}$, that is $\bigcap\limits_{\alpha<\omega}U_{\alpha}=\O$, where $U_{\alpha}=X\setminus A_{\alpha}$. Taking into account that the sequence $(U_{\alpha}:\alpha< \omega)$ is decreasing and $\omega$ is the first ordinal with cardinality $\aleph$, we obtain that 
$\bigcap\limits_{\alpha\in I}U_{\alpha}=\O$ for every set of ordinals $I\subseteq[0,\omega)$ with $|I|=\aleph$. Thus, $\aleph$ is not caliber of $X$.
\end{proof}

Thus, the following question naturally arises. It is an analog of Question 3.6 from [14].

\begin{question}\label{q:4.3} Let $X$ be a topological space for which every uncountable cardinal is caliber. Is every compact $Y\subseteq C_p(X)$ metrizable?\end{question}

Note that analogously as in [14, Theorem 3.7] it can proved using Theorem 4.2 that this question has a positive answer in the case $\aleph_2=2^{\aleph_1}$.

\begin{proposition}\label{p:4.4} Let $Y$ be a compact space such that for every set $B\subseteq Y$ with $|B|\leq \aleph_1$ the compact set $\overline{B}$ is metrizable. Then $Y$ is metrizable.\end{proposition}

\begin{proof} It is enough to prove that $Y$ is separable. Suppose that $Y$ is not separable. Then it easy to construct a sequence $(y_{\xi}: \xi<\aleph_1)$ of points $y_{\xi}\in Y$ such that $y_{\alpha}\not \in \overline{\{y_{\xi}:\xi<\alpha\}}$ for every $\alpha<\aleph_1$. Since the space $Z=\overline{\{y_{\xi}:\xi<\aleph_1\}}$ is metrizable, there exists an at most countable subset of the set $\{y_{\xi}:\xi<\aleph_1\}$ which is dense in $Z$. But this contradicts to the choice of $y_{\xi}$.\end{proof}

\begin{theorem}\label{th:4.5} Let $\aleph_i$ be the first uncountable cardinal with ${\rm cof}(\aleph_i)=\aleph_0$. Suppose that $\aleph_i>2^{\aleph_1}$. Then for every topological space $X$ for which every regular cardinal $\aleph\in [\aleph_1, 2^{\aleph_1}]$ is its caliber, every compact $Y\subseteq C_p(X)$ is metrizable.\end{theorem}

\begin{proof} For every compact space $Y\subseteq C_p(X)$ with $d(Y)\leq\aleph_1$ we have $w(Y)\leq 2^{\aleph_1}$. It follows from Theorem 4.2 and the condition $\aleph_i>2^{\aleph_1}$ that $w(Y)=\aleph_0$, that is $Y$ is metrizable. It remains yo use Proposition 4.4.
\end{proof}

The following result can be proved analogously.

\begin{theorem} \label{th:4.6} Let $\aleph_i$ be the first uncountable cardinal with ${\rm cof}(\aleph_i)=\aleph_0$ and $X$ be a topological space such that  $d(X)<\aleph_i$ and every regular cardinal $\aleph$ is a caliber of $X$. Then every compact $Y\subseteq C_p(X)$ is metrizable.
\end{theorem}

\begin{proof} Note that for every compact space $Y\subseteq C_p(X)$ and every dense in $X$ set $A$ the mapping $\varphi: Y\to C_p(A)$, $\varphi(y)=y|_A$, is a homeomorphic embedding. Therefore $w(Y)\leq d(X)<\aleph_i$ and according to Theorem 4.2, $Y$ is metrizable.
\end{proof}

\bibliographystyle{amsplain}

\end{document}